\documentclass[12pt]{article}
\usepackage{amssymb}
\usepackage{amsmath,amsthm}
\usepackage[latin1]{inputenc}
\usepackage{graphicx}
\usepackage{hyperref}
\usepackage{enumerate}
\usepackage{tikz}
\usepackage{tkz-graph}
\usepackage{mathrsfs}
\usepackage{verbatim}

\hypersetup{colorlinks=true, linkcolor=blue, citecolor=blue, urlcolor=blue}

\newtheorem{remark}{Remark}

\newtheorem{lemma}[remark]{Lemma}
\newtheorem{theorem}[remark]{Theorem}
\newtheorem{proposition}[remark]{Proposition}
\newtheorem{corollary}[remark]{Corollary}

\newcommand{\adim}{\operatorname{\mathrm{adim}}}

\title{The local metric dimension of the lexicographic product of graphs}

\author{G. A. Barrag\'{a}n-Ram\'{i}rez,  A. Estrada-Moreno, Y.  Ram\'{i}rez-Cruz
\\ and J. A. Rodr\'{i}guez-Vel\'{a}zquez\\
{\small Departament d'Enginyeria Inform\`atica i Matem\`atiques,}\\
{\small Universitat Rovira i Virgili,}  {\small Av. Pa\"{\i}sos
Catalans 26, 43007 Tarragona, Spain.} 
}

\begin{document}
\maketitle

\begin{abstract}
The metric dimension is quite a well-studied graph parameter.
Recently, the adjacency  dimension  and the local metric dimension have been introduced and studied.
In this paper, we give a general formula for the local metric dimension of the lexicographic product $G \circ \mathcal{H}$ of a connected graph $G$ of order $n$ and a family  $\mathcal{H}$ composed by $n$ graphs. 
We show that the local metric dimension 
of $G \circ \mathcal{H}$ can be expressed in terms of the true twin equivalence classes of $G$ and the local adjacency dimension of the graphs in 
$\mathcal{H}$. 
\end{abstract}

{\it Keywords:}  Local metric dimension; local adjacency dimension; lexicographic product graphs.

{\it AMS Subject Classification numbers:}   05C12; 05C76

\section{Introduction}
\label{sectIntro}

A metric generator of a metric space $(X,d)$ is a set $S\subset X$ of points in the space  with the property that every point of $X$  is uniquely determined by the distances from the elements of $S$, \textit{i.e}., 
for every $x,y\in X$, there exists $z\in S$ such that $d(x,z)\ne d(y,z)$ \cite{MR0268781}. In this case we say that $z$ \emph{distinguishes} the pair $x,y$.

 Given a simple and connected graph $G=(V,E)$, we consider the function $d_G:V\times V\rightarrow \mathbb{N}\cup \{0\}$, where $d_G(x,y)$ is the length of a shortest path between $u$ and $v$ and $\mathbb{N}$ is the set of positive integers. Then $(V,d_G)$ is a metric space since $d_G$ satisfies $(i)$ $d_G(x,x)=0$  for all $x\in V$,$(ii)$  $d_G(x,y)=d_G(y,x)$  for all $x,y \in V$ and $(iii)$ $d_G(x,y)\le d_G(x,z)+d_G(z,y)$  for all $x,y,z\in V$. 
A set $S\subset V$ is said to be a \emph{metric generator} for $G$ if any pair of vertices of $G$ is
distinguished by some element of $S$. A minimum cardinality metric generator is called a \emph{metric basis}, and
its cardinality the \emph{metric dimension} of $G$, denoted by $\dim(G)$.  

The notion of metric dimension of a graph was introduced by Slater in \cite{Slater1975}, where metric generators were called \emph{locating sets}. Harary and Melter independently introduced the same concept in  \cite{Harary1976}, where metric generators were called \emph{resolving sets}. Applications of this invariant to the navigation of robots in networks are discussed in \cite{Khuller1996} and applications to chemistry in \cite{Johnson1993,Johnson1998}.  Several variations of metric generators, including resolving dominating sets \cite{Brigham2003}, independent resolving sets \cite{Chartrand2003}, local metric sets \cite{Okamoto2010}, strong resolving sets \cite{Sebo2004}, adjacency resolving sets  \cite{JanOmo2012}, $k$-metric/adjacency generators \cite{Estrada-Moreno2014a,Estrada-Moreno2013}, simultaneous (strong) metric generators \cite{EstrGarcRamRodr2015,Ramirez-Cruz-Rodriguez-Velazquez_2014}, etc., have since been introduced and studied. 

 A set $S$ of vertices
in a connected graph $G$ is a \emph{local metric generator} for $G$ (also called local metric set for $G$ \cite{Okamoto2010}) if every two adjacent vertices of $G$ are distinguished by some vertex
of $S$. A minimum local metric generator is called a \emph{local metric
basis} for $G$ and its cardinality, the \emph{local metric dimension} of G, is  denoted by $\dim_l(G)$.

The concept of adjacency generator\footnote{Adjacency generators were called adjacency resolving sets in   \cite{JanOmo2012}.} was introduced by Jannesari and Omoomi in \cite{JanOmo2012} as a tool to study the metric dimension of lexicographic product graphs. A set $S\subset V$ of vertices in a graph $G=(V,E)$ is said to be  an \emph{adjacency generator} for $G$  if for every two  vertices $x,y\in V-S$ there exists $s\in S$ such that $s$ is adjacent to exactly one of $x$ and $y$. A minimum cardinality  adjacency generator is called an  \emph{adjacency basis} of $G$, and its cardinality  the  \emph{adjacency dimension} of $G$,  denoted by  $\adim(G)$ \cite{JanOmo2012}.
The concepts of \emph{local adjacency generator}, \textit{local adjacency basis} and \emph{local adjacency dimension} are defined by analogy, and the 
local adjacency dimension of a graph $G$ is denoted by $\adim_l(G)$. This concept has been studied further by Fernau and Rodr\'{i}guez-Vel\'{a}zquez in \cite{RV-F-2013,MR3218546} where they introduced the study of local adjacency generators and showed
that the (local) metric dimension of the corona product of a graph of order $n$ and some
non-trivial graph $H$ equals $n$ times the (local) adjacency dimension of $H$. As a consequence of this strong relation they showed that the problem of computing the local metric dimension and the (local) adjacency dimension of a graph  is
NP-hard.
%Since any adjacency basis is a metric generator, $\dim(G)\le \adim (G)$.  Besides, for any connected graph $G$ of diameter at most two, $\adim(G)=\dim(G)$.

As pointed out in \cite{RV-F-2013,MR3218546}, 
any  adjacency generator for a graph $G=(V,E)$ is  also a metric generator in a suitably chosen metric space.
%%%%%%%%%%%%%%%%%%%%%%%%%%%%%%%%%
Given a positive integer $t$,  we define the distance function $d_{G,t}:V\times V\rightarrow \mathbb{N}\cup \{0\}$, where
\begin{equation}\label{distinguishAdj}
d_{G,t}(x,y)=\min\{d_G(x,y),t\}.%\tag{$\bullet$}.
\end{equation}
Then any metric generator for $(V,d_{G,t})$ is a metric generator for $(V,d_{G,t+1})$ and, as a consequence, the metric dimension of $(V,d_{G,t+1})$  is less than or equal to the metric dimension of $(V,d_{G,t})$. In particular, the metric dimension of $(V,d_{G,1})$ equals $|V|-1$,  the metric dimension of $(V,d_{G,2})$ equals $\adim(G)$ and, if $G$ has diameter $D(G)$, then $d_{G,D(G)}=d_G$ and so  the metric dimension of  $(V,d_{G,D(G)})$  equals $\dim(G)$.
Notice that when using the metric $d_{G,t}$ 
the concept of metric generator needs not be restricted to the case of connected graphs, as for any pair of vertices $x,y$ belonging to different connected components of $G$ we can assume that $d_G(x,y)=+\infty$ and so $d_{G,t}(x,y)=t$. % for any $t$ greater than or equal to the maximum diameter of a connected component of $G$. 

Notice that $S$ is an adjacency generator for $G$ if and only if $S$ is an adjacency generator for its complement $\overline{G}$. 
This is justified by the fact that given an adjacency generator $S$ for $G$, it holds that for every $x,y\in V- S$ there exists $s\in S$ such that $s$ is adjacent to exactly one of $x$ and $y$, and this property holds in $\overline{G}$. Thus, $\adim(G)=\adim (\overline{G}).$

From the definitions of the different variants of generators, we can observe: an adjacency generator is a metric generator; a metric generator is a  local metric generator; a local adjacency generator is a local metric generator; and an  adjacency generator is a local adjacency generator.
These facts show that 
the following inequalities hold for any graph $G$:
\begin{enumerate}[{\rm (i)}]
 \item $\dim(G)\leq \adim(G)$;
 \item $\dim_l(G)\leq \dim(G)$;
 \item $\dim_l(G)\leq \adim_l(G)$;
 \item $\adim_l(G)\leq \adim (G)$.
\end{enumerate}
Moreover, if $D(G)\le 2$, then $\dim(G)= \adim(G)$ and $\dim_l(G)=\adim_l(G)$.

The radius of a  graph $G$ is denoted by $r(G)$. The following result describes situations with very small or large local adjacency dimensions.

\begin{theorem}{\rm \cite{RV-F-2013}}\label{LocalAdjDimSmall-Large}
Let $G$ be a non-empty graph of order $t$. The following assertions hold.
\begin{enumerate}[{\rm (i)}]

\item 
$\adim_l(G)=1$ if and only if $G$ is a bipartite graph having only one non-trivial connected component $G^*$ and $r(G^*)\le 2$.
 
 \item $\adim_l( G)= t-1$ if and only if $G\cong K_{t}$.
\end{enumerate}
\end{theorem}

The remainder of the paper is structured as follows. After introducing some useful notation and terminology in Section~\ref{sectPrelim}, we extensively discuss our main results in Section~\ref{sectMain}. Then, Section~\ref{sectAdjvsK1} is devoted to show how all previous results presented in Section~\ref{sectMain} in terms of the local adjacency dimension can be expressed in terms of the local metric dimension of graphs of the form $K_1+H$. Finally, we show in Section~\ref{sectAdjProduct} that the methodology used for studying the local metric dimension can be applied to the case of the local adjacency dimension of lexicographic product graphs. In particular, we discuss when the values of both dimensions coincide.

\section{Preliminary concepts}\label{sectPrelim}
Throughout the paper, we will use the notation $K_n$, $K_{r,n-r}$, $C_n$, $N_n$ and $P_n$ for complete graphs, complete bipartite graphs, cycle graphs, empty graphs and path graphs of order $n$, respectively. 

We use the notation $u \sim v$ if $u$ and $v$ are adjacent and $G \cong H$ if $G$ and $H$ are isomorphic graphs. For a vertex $v$ of a graph $G$, $N_G(v)$ will denote the set of neighbours or \emph{open neighbourhood} of $v$ in $G$, \emph{i.e.} $N_G(v)=\{u \in V(G):\; u \sim v\}$. 
%We denote by $\delta(v)=|N(v)|$ the degree of vertex $v$, as well as $\delta(G)=\min_{v \in V(G)}\{\delta(v)\}$ and $\Delta(G)=\max_{v \in V(G)}\{\delta(v)\}$.

The \emph{closed neighbourhood} of $v$, denoted by $N_G[v]$, equals $N_G(v) \cup \{v\}$. If there is no ambiguity, we will simply write $N(v)$ or $N[v]$. Two vertices $x,y\in V(G)$  are \textit{true twins} in $G$ if $N_G[x]=N_G[y]$. 

 The subgraph of $G$ induced by a set $S$ of vertices is denoted  by $\langle S\rangle_G$. If there is no ambiguity, we will simply write $\langle S\rangle$.  The length of a shortest cycle (if any) in a graph $G$ is called the girth of $G$, and it is denoted by $\mathtt{g}(G)$. Acyclic graphs are considered to have infinite girth.

From now on we denote by $\mathcal{G}$ the set of graphs $H$ satisfying  that for every local adjacency basis $B$, there exists $v\in V(H)$ such that $B\subseteq N_H(v)$. Notice that  the only local adjacency basis of an empty graph $N_r$ is the empty set, and so  $N_r\in \mathcal{G}$. Moreover, $K_1\cup K_2\in \mathcal{G}$. In fact, a non-connected graph $H\in \mathcal{G}$ if and only if $H\cong N_r$ or $H\cong N_r\cup G$, where $G$ is a connected graph in $\mathcal{G}$. We denote by $\Phi$ the family of empty graphs. Notice that  $  \Phi\subset \mathcal{G}$.
On the other hand, it is readily seen that no graph of radius greater than or equal to four  belongs to  $\mathcal{G}$.  As we will see in Proposition 
\ref{lemmaGirth}, if $H\in \mathcal{G}$  is  a connected  graph different from a tree, then $\mathtt{g}(H)\le 6$.

\subsection{The lexicographic product $G \circ \mathcal{H}$}

Let $G$ be a graph of order $n$, and let $\mathcal{H}=\{H_1,H_2,\ldots,H_n\}$ be an ordered family composed by $n$ graphs. The \emph{lexicographic product} of $G$ and $\mathcal{H}$ is the graph $G \circ \mathcal{H}$, such that $V(G \circ \mathcal{H})=\bigcup_{u_i \in V(G)} (\{u_i\} \times V(H_i))$ and $(u_i,v_r)(u_j,v_s) \in E(G \circ \mathcal{H})$ if and only if $u_iu_j \in E(G)$ or $i=j$ and $v_rv_s \in E(H_i)$. 
Figure \ref{figExampleOfLexiFamily} shows the lexicographic product of $P_3$ and the family composed by $\{P_4,K_2,P_3\}$, and the lexicographic product of $P_4$ and the family $\{H_1,H_2,H_3,H_4\}$, where $H_1 \cong H_4 \cong K_1$ and $H_2 \cong H_3 \cong K_2$. In general, we can construct the graph $G\circ\mathcal{H}$ by taking one copy of each $H_i\in\mathcal{H}$ and joining by an edge every vertex of $H_i$ with every vertex of $H_j$ for every $u_i u_j\in E(G)$. Note that $G\circ\mathcal{H}$ is connected if and only if $G$ is connected.

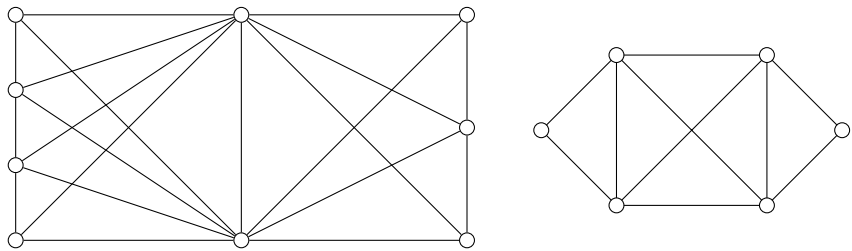
\begin{figure}[!ht]
\centering
\begin{tikzpicture}[transform shape, inner sep = .7mm]
%%El espaciamiento entre los vértices de arriba
\pgfmathsetmacro{\espacio}{1};
\node [draw=black, shape=circle, fill=white] (v5) at (3,0) {};
\node [draw=black, shape=circle, fill=white] (v6) at (3,3*\espacio) {};
\draw[black] (v5) -- (v6);
\foreach \ind in {1,...,4}
{
\pgfmathsetmacro{\yc}{(\ind-1)*\espacio};
\node [draw=black, shape=circle, fill=white] (v\ind) at (0,\yc) {};
\draw[black] (v5) -- (v\ind);
\draw[black] (v6) -- (v\ind);
\ifthenelse{\ind>1}
{
\pgfmathtruncatemacro{\bind}{\ind-1};
\draw[black] (v\ind) -- (v\bind);
}
{};
}
\node [draw=black, shape=circle, fill=white] (v7) at (6,0) {};
\node [draw=black, shape=circle, fill=white] (v8) at (6,1.5*\espacio) {};
\node [draw=black, shape=circle, fill=white] (v9) at (6,3*\espacio) {};
\draw[black] (v7) -- (v8) -- (v9);
\foreach \ind in {7,...,9}
{
\draw[black] (v5) -- (v\ind);
\draw[black] (v6) -- (v\ind);
}
\end{tikzpicture}
\hspace{0.5cm}
\begin{tikzpicture}[transform shape, inner sep = .7mm]

\draw(-1,-1) -- (1,1)--(1,-1)--(-1,1)--(-2,0)--(-1,-1)--(1,-1)--(2,0)--(1,1)--(-1,1)--(-1,-1);
\node  at  (0,-1.5) {};

\node [draw=black, shape=circle, fill=white] at  (-1,-1) {};
\node [draw=black, shape=circle, fill=white] at  (-1,1) {};
\node [draw=black, shape=circle, fill=white] at  (1,-1) {};
\node [draw=black, shape=circle, fill=white] at  (1,1) {};
\node [draw=black, shape=circle, fill=white] at  (-2,0) {};
\node [draw=black, shape=circle, fill=white] at  (2,0) {};

%\node  at  (-1,-1.3) {$u_2$};
%\node  at  (-1,1.3) {$u_3$};
%\node  at  (1,-1.3) {$u_4$};
%\node  at  (1,1.3) {$u_5$};
%\node  at  (-2.4,0) {$u_1$};
%\node  at  (2.4,0) {$u_6$};
\end{tikzpicture} 
\caption{The lexicographic product graphs $P_3\circ\{P_4,K_2,P_3\}$ and $P_4\circ\{H_1,H_2,H_3,H_4\}$, where $H_1 \cong H_4 \cong K_1$ and $H_2 \cong H_3 \cong K_2$.}\label{figExampleOfLexiFamily}
\end{figure}

In particular, if $H_i \cong H$ for every $H_i \in \mathcal{H}$, then $G \circ \mathcal{H}$  is a standard lexicographic product graph, 
 which is denoted as $G \circ H$ for simplicity.
Another particular case of lexicographic product graphs is the join graph. The \emph{join graph} $G+H$\label{g join} is defined as the graph obtained from disjoint graphs $G$ and $H$ by taking one copy of $G$ and one copy of $H$ and joining by an edge each vertex of $G$ with each vertex of $H$ \cite{Harary1969,Zykov1949}. Note that $G+H\cong K_2\circ\{G,H\}$. The join operation is commutative and associative. Now, for the sake of completeness, Figure \ref{ex join} illustrates two examples of join graphs.

\begin{figure}[h]
\centering
\begin{tabular}{cccccc}
  %\hline
\begin{tikzpicture}[transform shape, inner sep = .7mm]

\draw(0,0) -- (0,3);
\draw(1.5,1.5) -- (3,0.5) -- (3,2.5) -- cycle;
\draw(1.5,1.5) -- (0,0);
\draw(3,0.5) -- (0,0);
\draw(3,2.5) -- (0,0);
\draw(1.5,1.5) -- (0,1);
\draw(3,0.5) -- (0,1);
\draw(3,2.5) -- (0,1);
\draw(1.5,1.5) -- (0,2);
\draw(3,0.5) -- (0,2);
\draw(3,2.5) -- (0,2);
\draw(1.5,1.5) -- (0,3);
\draw(3,0.5) -- (0,3);
\draw(3,2.5) -- (0,3);

\node [draw=black, shape=circle, fill=white] at  (0,0) {};
\node [draw=black, shape=circle, fill=white] at  (0,1) {};
\node [draw=black, shape=circle, fill=white] at  (0,2) {};
\node [draw=black, shape=circle, fill=white] at  (0,3) {};
\node [draw=black, shape=circle, fill=white] at  (1.5,1.5) {};
\node [draw=black, shape=circle, fill=white] at  (3,0.5) {};
\node [draw=black, shape=circle, fill=white] at  (3,2.5) {};

\end{tikzpicture} & & \hspace*{0.7cm} & &
\begin{tikzpicture}[transform shape, inner sep = .7mm]
\draw(0,0) -- (4,0);
\draw(0,1.5) -- (4,1.5);
\draw(0,0) -- (4,1.5);
\draw(0,1.5) -- (4,0);
\draw(1.2,3) -- (0,0);
\draw(1.2,3) -- (0,1.5);
\draw(1.2,3) -- (4,0);
\draw(1.2,3) -- (4,1.5);
\draw(2.8,3) -- (0,0);
\draw(2.8,3) -- (0,1.5);
\draw(2.8,3) -- (4,0);
\draw(2.8,3) -- (4,1.5);

\node [draw=black, shape=circle, fill=white] at  (0,0) {};
\node [draw=black, shape=circle, fill=white] at  (0,1.5) {};
\node [draw=black, shape=circle, fill=white] at  (4,0) {};
\node [draw=black, shape=circle, fill=white] at  (4,1.5) {};
\node [draw=black, shape=circle, fill=white] at  (1.2,3) {};
\node [draw=black, shape=circle, fill=white] at  (2.8,3) {};

\end{tikzpicture} \\
  %\hline
\end{tabular}
\caption{Two join graphs:  $P_4+C_3 \cong K_2\circ \{P_4,C_3\}$ and $N_2+N_2+N_2 \cong K_3\circ N_2$.}
\label{ex join}
\end{figure}
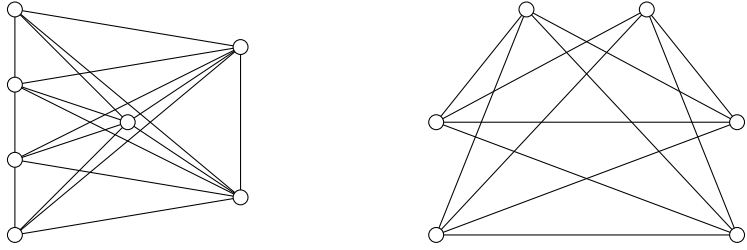

Moreover, complete $k$-partite graphs,  $K_{p_1,p_2,...,p_k}\cong K_n \circ \{N_{p_1},N_{p_2},\dots ,N_{p_k}\}\cong N_{p_1}+N_{p_2}+\cdots +N_{p_k}$, are typical examples of join graphs. The particular case   illustrated in Figure \ref{ex join} (right hand side), is no other than the complete $3$-partite graph $K_{2,2,2}$.

The relation between distances in a lexicographic product graph and those in its factors is presented in the following remark, for which it is necessary to recall (\ref{distinguishAdj}).

\begin{remark}\label{remarkDistLexi}
 If $G$ is a connected graph and $(u_i,b)$ and $(u_j,d)$ are vertices of $G\circ \mathcal{H}$, then
$$d_{G\circ \mathcal{H}}((u_i,b),(u_j,d))=\left\{\begin{array}{ll}

                            d_G(u_i,u_j), & \mbox{if $i\ne j$,} \\
                            \\
                           d_{H_i,2}(b,d), & \mbox{if $i=j$.}
                            \end{array} \right.\\
$$
\end{remark}
We would point out that the remark above was stated in  \cite{Hammack2011,Imrich2000} for the case where $H_i\cong H$ for all $H_i\in \mathcal{H}$.

The lexicographic product has been studied from different points of view in the literature. One of the most common researches focuses on finding relationships between the value of some invariant in the product and that of its factors. In this sense, we can find in the literature a large number of investigations on diverse topics. For instance, the metric dimension and related parameters have been studied in \cite{Estrada-Moreno2014b,Feng2012a, JanOmo2012, Kuziak2014, Ramirez_Estrada_Rodriguez_2015, Saputro2013}. For more information on product graphs we suggest the books \cite{Hammack2011,Imrich2000}.
  
  In order to state our main result (Theorem \ref{DimLocalLexicInTermsAdjaLoc}) we need to introduce some additional notation.  
Let $\mathcal{U}=\{U_1,U_2,\dots , U_k\}$ be the set of non-singleton true twin equivalence classes of a graph $G$. For the remainder of this paper we will assume that $G$ is connected and has order $n\ge 2$, and $\mathcal{H}=\{H_1,\ldots,H_n\}$. We now define the following sets and parameters:

\begin{itemize}
\item $T(G)=\bigcup_{j=1}^kU_j$.
\item $V_E=\left\{u_i\in V(G)-T(G) :\, H_i \in \Phi \right\}$.
\item $I=\{u_i\in V(G):\, H_i\in \mathcal{G}\}$.
\item For any $I_j=I\cap U_j\ne \emptyset $, we can choose some  $u\in I_j$ and set  $I'_j=I_j-\{u\}$. We define the set $X_E=I-\bigcup_{I_j'\ne \emptyset}I_j'$.
\item We say that two vertices $u_i,u_j \in X_E$ satisfy the relation $\mathcal{R}$ if and only if $u_i \sim u_j$ and $d_G(u,u_i)=d_G(u,u_j)$ for all $u\in V(G)-(V_E \cup \{u_i,u_j\})$.
\item We define $\mathcal{A}$ as the family of sets $A\subseteq X_E$  such that for every pair of vertices $u_i,u_j\in X_E$ satisfying $\mathcal{R}$ there exists a vertex in $A$ that distinguishes them.
\item $\varrho(G,\mathcal{H})=\displaystyle\min_{A\in \mathcal{A}}\left\{|A|\right\}.$  
\end{itemize}

%\newpage 

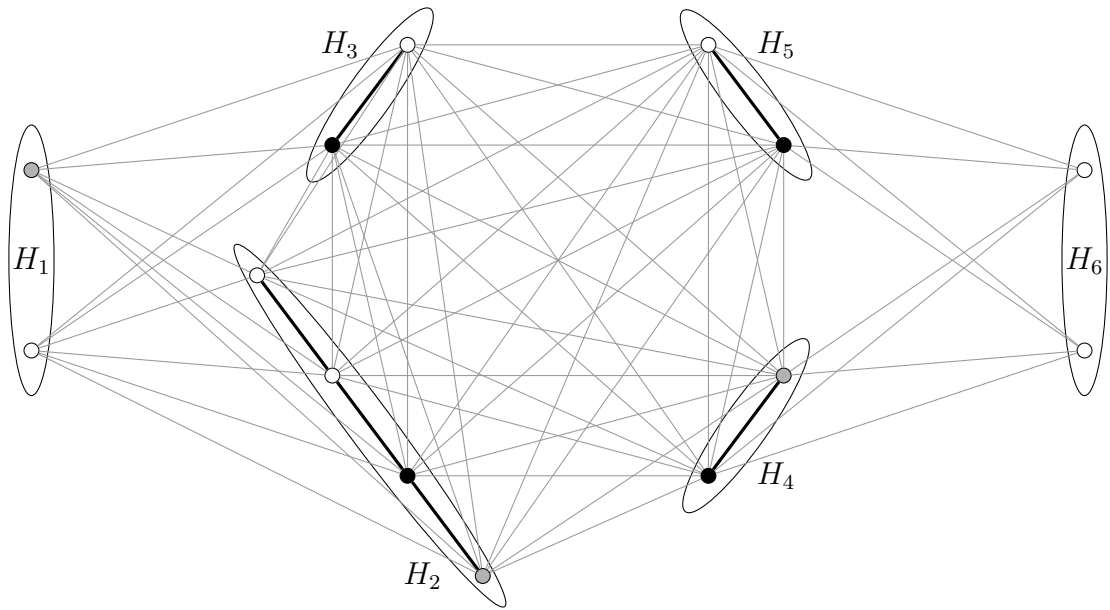
\begin{figure}[h]
\centering
\begin{tikzpicture}[transform shape, inner sep = .7mm]

\node [draw=black, shape=circle, fill=black!30] (21) at  (-1,-4.2) {};  
\node [draw=black, shape=circle, fill=black] (22) at  (-2,-4.2+4/3) {}; 
\node [draw=black, shape=circle, fill=white] (23) at  (-3,-4.2+8/3) {};
\node [draw=black, shape=circle, fill=white] (24) at  (-4,-0.2) {};

\node [draw=black, shape=circle, fill=white] (31) at  (-2,4.2-4/3) {}; 
\node [draw=black, shape=circle, fill=black] (32) at  (-3,4.2-8/3) {};

\node [draw=black, shape=circle, fill=white] (51) at  (2,4.2-4/3) {}; 
\node [draw=black, shape=circle, fill=black] (52) at  (3,4.2-8/3) {};

\node [draw=black, shape=circle, fill=black] (41) at  (2,-4.2+4/3) {}; 
\node [draw=black, shape=circle, fill=black!30] (42) at  (3,-4.2+8/3) {};

\node [draw=black, shape=circle, fill=white] (11) at  (-7,-1.2) {}; 
\node [draw=black, shape=circle, fill=black!30] (12) at  (-7,1.2) {};

\node [draw=black, shape=circle, fill=white] (61) at  (7,-1.2) {}; 
\node [draw=black, shape=circle, fill=white] (62) at  (7,1.2) {};

\draw[line width=1.2pt] (21) -- (22)  -- (23) -- (24);
\pgfmathsetmacro{\angle}{atan(-4/3.2)+88};
\draw[rotate around={\angle:(-5/2,-2.2)}] (-5/2,-2.2) ellipse (0.3 and 3);

\pgfmathsetmacro{\angle}{atan(4.6/3)+88};
\draw[rotate around={\angle:(5/2,-13.2/6)}] (5/2,-13.2/6) ellipse (0.3 and 1.4);
\draw[rotate around={\angle:(-5/2,13.2/6)}] (-5/2,13.2/6) ellipse (0.3 and 1.4);
\pgfmathsetmacro{\angle}{atan(4.6/3)+160};
\draw[rotate around={\angle:(5/2,13.2/6)}] (5/2,13.2/6) ellipse (0.3 and 1.4);

\draw (-7,0) ellipse (0.3 and 1.8);
\draw (7,0) ellipse (0.3 and 1.8);
\node at  (-7,0) {$H_1$};
\node at  (7,0) {$H_6$};
\node at  (-1.8,-4.2) {$H_2$}; 
\node at  (-2.9,4.2-4/3) {$H_3$};
\node at  (2.9,4.2-4/3) {$H_5$};
\node at  (2.9,-4.2+4/3) {$H_4$};

\foreach \ind in {3,...,5}
{
\pgfmathtruncatemacro{\x}{1};
\pgfmathtruncatemacro{\y}{2};
\draw[line width=1.2pt] (\ind\x) -- (\ind\y);
}

\foreach \ind in {1,4,5}
{
\foreach \indy in {1,2}
{
\foreach \x in {1,2}
{
\draw[black!40] (3\x) -- (\ind\indy);
}
}
}

\foreach \ind in {5,6}
{
\foreach \indy in {1,2}
{
\foreach \x in {1,2}
{
\draw[black!40] (4\x) -- (\ind\indy);
}
}
}

\foreach \indy in {1,2}
{
\foreach \x in {1,2}
{
\draw[black!40] (5\x) -- (6\indy);
}
}

\foreach \ind in {1,3,4,5}
{
\foreach \indy in {1,2}
{
\foreach \x in {1,...,4}
{
\draw[black!40] (2\x) -- (\ind\indy);
}
}
}
\end{tikzpicture} 
\caption{The graph $G\circ{\cal H}$, where $G$ is the right-hand graph shown in Figure \ref{figExampleOfLexiFamily} and $\mathcal{H}$ is the family composed by the graphs $H_1\cong H_6\cong N_2$, $H_2\cong P_4$, $H_3\cong H_4\cong H_5\cong K_2$. The set of black- and grey-coloured vertices is a local metric basis of $G\circ{\cal H}$.}
\label{Ex notation Lexicografico}
\end{figure}

With the aim of clarifying what  this notation means, we proceed to show an example where we explain the role of these parameters when constructing a local metric generator $W$ for a lexicographic product graph. Let $G$ be the right-hand graph shown in Figure \ref{figExampleOfLexiFamily} and let $\mathcal{H}$ be the family composed by the graphs $H_1\cong H_6\cong N_2$, $H_2\cong P_4$, $H_3\cong H_4\cong H_5\cong K_2$. Figure \ref{Ex notation Lexicografico} shows the graph $G\circ{\cal H}$. Consider any $H_i\notin\Phi$. Note that the restriction of any local metric basis of $G\circ{\cal H}$ to the vertices of $\langle \{u_i\} \times V(H_i) \rangle \cong H_i$ must be a local adjacency generator for $\langle \{u_i\} \times V(H_i) \rangle$, as two adjacent vertices of $\langle \{u_i\} \times V(H_i) \rangle$ are not distinguished by any vertex outside ${u_i}\times V(H_i)$, so we can assume that the black-coloured vertices belong to $W$. Moreover, $U_1=\{u_2,u_3\}$ and $U_2=\{u_4,u_5\}$ are the non-singleton true twin equivalence classes of $G$. Since $u_4,u_5 \in I\cap U_2$, we have that no pair of non-black-coloured vertices in $({u_4}\times V(H_4))\cup({u_5}\times V(H_5))$ is distinguished by any black-coloured vertex, so we add to $W$ the grey-coloured vertex corresponding to the copy of $H_4$ and, by analogy, we add to $W$ the grey-coloured vertex corresponding to the copy of $H_2$. Besides, note that the white-coloured vertices of the copies of $H_3$ and $H_5$ are only distinguished by themselves and by vertices from the copies of $H_1$ and $H_6$, so we need to add one more vertex to $W$, e.g. the grey-coloured vertex in the copy of $H_1$. Note that, according to our previous definitions, we have $V_E=\{u_1,u_6\}$   and we take $I'_1= \{u_2\}$ and $I'_2=\{u_4\}$. Thus, $X_E=\{u_1,u_3,u_5,u_6\}$. Therefore, since  $u_1\in X_E$ distinguishes the pair $u_3,u_5$, the sole pair of vertices from $X_E$ satisfying $\mathcal{R}$, we take $A=\{u_1\}$ and conclude that $\varrho(G,\mathcal{H})=1$. Notice that, $\displaystyle\sum_{i=1}^{6}\adim_{l}(H_i)=4$, $\displaystyle\sum_{I\cap U_j\ne \emptyset}(|I\cap U_j|-1)=2$ and 
$\dim_l(G\circ\mathcal{H})=\displaystyle\sum_{i=1}^{6}\adim_{l}(H_i)+
\sum_{I\cap U_j\ne \emptyset}(|I\cap U_j|-1)+\varrho(G,\mathcal{H})=7.$

%For the remainder of the paper, definitions will be introduced whenever a concept is needed.

\section{Main results}\label{sectMain}

\begin{theorem}\label{DimLocalLexicInTermsAdjaLoc}
Let  $G$ be a connected graph     of order $n\ge 2$, let $\{U_1,U_2,\dots , U_k\}$ be the set of non-singleton true twin equivalence classes of $G$ and let $\mathcal{H}=\{H_1,\ldots,H_n\}$ be a family     of graphs. 
Then
$$\dim_l(G\circ\mathcal{H})=\displaystyle\sum_{i=1}^{n}\adim_{l}(H_i)+
\sum_{I\cap U_j\ne \emptyset}(|I\cap U_j|-1)+\varrho(G,\mathcal{H}).$$
\end{theorem}

\begin{proof}
 We will first construct a local metric generator for $G\circ\mathcal{H}$. To this end, we need to introduce some notation. Let $V(G)=\{u_1,\ldots,u_n\}$ and let $S_i$ be a local adjacency basis of $H_i$, where $i\in \{1,\dots, n\}$. For any $I_j=I\cap U_j\ne \emptyset $, we choose $u\in I_j$ and set  $I'_j=I_j-\{u\}$. Now, for every  $u_i\in I'_j\ne \emptyset $, let $v_i\in V(H_i)$ such that $S_i\subseteq N_{H_i}(v_i)$. Finally, we consider a set $A\subseteq X_E$ achieving the minimum in the definition of $\varrho(G,\mathcal{H})$ and, for each $u_i\in A$, we choose one vertex $y_i\in V(H_i)-S_i$ such that $S_i\subseteq N_{H_i}(y_i)$. 
 We claim that the set 
  $$S=\left( \bigcup_{S_i\ne \emptyset} (\{u_i\}\times S_i) \right) \cup \left( \bigcup_{I'_j\ne \emptyset }\{(u_i,v_i):\;  u_i\in I'_j  \}\right)\cup \left( \bigcup_{u_i\in A} \{(u_i,y_i)\}\right)$$
is a  local metric generator for $G\circ\mathcal{H}$. We differentiate the following four cases for two adjacent vertices  $(u_i,v),$ $(u_j,w)\in V(G\circ\mathcal{H})-S$.\\
\\Case 1. $i=j$. In this case $v\sim w$. Since  $S_i$ is a local adjacency basis of $H_i$,  there exists $x\in S_i$ such that $d_{H_i,2}(x,v)\ne d_{H_i,2}(x,w)$ and so
for $(u_i,x)\in \{u_i\}\times S_i\subset S$ we have  $d_{G\circ\mathcal{H}}((u_i,x),(u_i,v))= d_{H_i,2}(x,v)\ne d_{H_i,2}(x,w)= d_{G\circ\mathcal{H}}((u_i,x),(u_i,w))$.\\
\\Case 2. $i\ne j$, $u_i,u_j\in U_l$ and $u_i\not \in I_l $. 
For any  $y\in S_i-N_{H_i}(v)$ we have that 
 $(u_i,y)\in \{u_i\}\times S_i\subseteq S$  and $d_{G\circ\mathcal{H}}((u_i,y),(u_i,v)) =2\ne 1= d_{G\circ\mathcal{H}}((u_i,y),(u_j,w))$. 
\\
\\Case 3. $i\ne j$,  $u_i,u_j\in U_l$  and $u_i , u_j \in I_l $.
If $v=v_i$ and $w=v_j$, then $(u_i,v_i)\in S$ or $(u_j,v_j)\in S$. If $v\ne v_i$ or $w\ne v_j$ (say $v\ne v_i$) then either $S_i \subseteq N_{H_i}(v)$, in which case $d_{G\circ\mathcal{H}}((u_i,v_i),(u_i,v)) = 2\ne 1= d_{G\circ\mathcal{H}}((u_i,v_i),(u_j,w))$, or there exists $y\in S_i-N_{H_i}(v)$ such that    $(u_i,y)\in \{ u_i\}\times S_i \subseteq S$ and 
$d_{G\circ\mathcal{H}}((u_i,y),(u_i,v)) = 2\ne 1= d_{G\circ\mathcal{H}}((u_i,y),(u_j,w))$.\\
\\Case 4. $i\ne j$ and $N_G[u_i]\ne N_G[u_j]$. Notice that, in this case, $u_i\sim u_j$. If $u_i\not \in I$, then $S_i\ne \emptyset$ and there exists $y\in S_i-N_{H_i}(v)$ such that    $(u_i,y)\in \{ u_i\}\times S_i \subseteq S$ and 
$d_{G\circ\mathcal{H}}((u_i,y),(u_i,v)) = 2\ne 1= d_{G\circ\mathcal{H}}((u_i,y),(u_j,w))$. Now, assume that   $u_i,u_j  \in I$. If $u_i\in I'_l$
 or $u_j\in I'_l$ for some $l$ (say $u_i\in I'_l$), then  $d_{G\circ\mathcal{H}}((u_i,v_i),(u_i,v))=2\ne 1=d_{G\circ\mathcal{H}}((u_i,v_i),(u_j,w)) $ or there exists $y\in S_i$ such that  $d_{G\circ\mathcal{H}}((u_i,y),(u_i,v))=2\ne 1=d_{G\circ\mathcal{H}}((u_i,y),(u_j,w)) $. Finally, if $u_i,u_j\notin \bigcup I'_l$, then by the construction of $S$ there exists $u_l\in A\cup (V(G)-X_E)$ such that $d_G(u_l,u_i)\ne d_G(u_l,u_j)$. Since $u_l\in \{x: (x,y)\in S\}$, there exists $y\in V(H_l)$ such that  $d_{G\circ\mathcal{H}}((u_l,y),(u_i,v))\ne d_{G\circ\mathcal{H}}((u_l,y),(u_j,w)) $.

In conclusion, $S$ is a local metric generator for $G\circ\mathcal{H}$ and, as a result, $$\dim_l(G\circ\mathcal{H})\le |S|=\displaystyle\sum_{i=1}^{n}\adim_{l}(H_i)+
\sum_{I_j\ne \emptyset}(|I_j|-1)+\varrho(G,\mathcal{H}).$$

It remains to show that $\displaystyle\dim_l(G\circ\mathcal{H})\ge \displaystyle\sum_{i=1}^{n}\adim_{l}(H_i)+
\sum_{I_j\ne \emptyset}(|I_j|-1)+\varrho(G,\mathcal{H}).$  To this end, we take a local metric basis $W$ of 
$G\circ\mathcal{H}$ and for every $u_i\in V(G)$ we define the set $W_i=\{y: (u_i,y)\in W\}$. As for any  $u_i\in V(G)$ and two adjacent vertices $v,w\in V(H_i)$, no vertex outside  $\{u_i\}\times W_i$  distinguishes $(u_i,v)$ and $(u_i,w)$, we can conclude  that $W_i$ is a local adjacency generator for $H_i$. Hence,
\begin{equation}\label{condAll}
|W_i|\ge \adim_l(H_i),\text{ for all }i\in \{1,\dots, n\}.
\end{equation} 

Now suppose, for the purpose of contradiction, that  there exist  $u_i,u_j\in I\cap U_l$ such that $|W_i|= \adim_l(H_i)$ and $|W_j|= \adim_l(H_i)$. In such a case, there exist 
$v_i\in V(H_i)-W_i$ and $v_j\in V(H_j)-W_j$ such that $W_i\subseteq N_{H_i}(v_i)$ and $W_j\subseteq N_{H_j}(v_j)$, which is a contradiction. Hence, if $|I\cap U_l|\ge 2$, then $|\{u_i\in I\cap U_l:\; |W_i|\ge \adim_l(H_i)+1\}|\ge |I\cap U_l|-1$ and, as a consequence,
\begin{equation}\label{condTT}
\sum_{u_i\in I\cap T(G)}|W_i|\ge \sum_{u_i\in I\cap T(G)}\adim_{l}(H_i)+
\sum_{I\cap U_j\ne \emptyset}(|I\cap U_j|-1).
\end{equation} 

%Now, suppose that there are two adjacent vertices  $u_i,u_j \textcolor{red}{\in I}$ not belonging to the same true twin equivalence class such that $|W_i|= \adim_l(H_i)$ and $|W_j|= \adim_l(H_j)$. In such a case, there exist 
%$v_i\in V(H_i)-W_i$ and $v_j\in V(H_j)-W_j$ such that $W_i\subseteq N_{H_i}(v_i)$ and $W_j\subseteq N_{H_j}(v_j)$, so there exists $u_l\in W_l$, $l\ne i,j$, such that $d_G(u_l,u_i)\ne d_G(u_l,u_j)$.

On the other hand, assume that $\varrho(G,\mathcal{H})\neq \emptyset$. We claim that
\begin{equation}\label{condVarrho}
\sum_{u_j\in  X_E}|W_j|\geq \sum_{u_j\in  X_E}\adim_l(H_j)+\varrho(G,\mathcal{H}).
\end{equation}

To see this, we will prove that for any pair of vertices $u_i,u_j$ satisfying $\mathcal{R}$ there exists $u_r\in X_E$ such that $|W_r|\geq \adim_l(H_r)+1$. If $|W_i|=\adim_l(H_i)+1$ or $|W_j|=\adim_l(H_j)+1$, then we are done. Suppose that $|W_i|=\adim_l(H_i)$ and $|W_j|=\adim_l(H_j)$. Since $W_i$ and $W_j$ are local adjacency bases of $H_i$ and $H_j$, respectively,  there exist $v\in V(H_i)$ and $w\in V(H_j)$ such that $\{u_i\}\times W_i\subseteq N_{\langle  \{u_i\}\times V(H_i)  \rangle}(u_i,v)$ and 
$\{u_j\}\times W_j\subseteq  N_{\langle  \{u_j\}\times V(H_j)  \rangle}(u_j,w)$. Thus, there exists $(u_r,y)\in \{u_r\}\times W_r$, $r\ne i,j$,   which distinguishes the pair $(u_i,v), (u_j,w)$, and so $d_G(u_r,u_i)\ne d_G(u_r,u_j)$. Hence, since $u_i,u_j$ satisfy $\mathcal{R}$, we can claim that $u_r\in V_E\subseteq X_E$ and so $|W_r|>0=\adim_l(H_r)$. In consequence, \eqref{condVarrho} holds.
Therefore, \eqref{condAll}, \eqref{condTT} and \eqref{condVarrho} lead to     
$$\displaystyle\dim_l(G\circ\mathcal{H})  =\sum_{i=1}^{n}|W_i|\ge \displaystyle\sum_{i=1}^{n}\adim_{l}(H_i)+
\sum_{I\cap U_j\ne \emptyset}(|I\cap U_j|-1)+\varrho(G,\mathcal{H}), 
$$
as required.
\end{proof}

From now on we proceed  to obtain some particular cases of this main result. 
 To begin with, we consider the case $\varrho(G,\mathcal{H})=0$.
 
 \begin{corollary}\label{Rocero}
 Let $G$ be a connected graph of order $n\ge 2$ and let $\mathcal{H}=\{H_1,\ldots,H_n\}$ be a family of graphs. If  for any pair of adjacent vertices $u_i,u_j\in V(G)$, not belonging to the same true twin equivalence class,  $H_i\notin \mathcal{G}$ or $H_j\notin \mathcal{G}$, or there exists $u_l\in V(G)$ such  that $H_l\notin \Phi$ and $d_G(u_l,u_i)\ne d_G(u_l,u_j)$, then $$\dim_l(G\circ\mathcal{H})=\displaystyle\sum_{i=1}^{n}\adim_{l}(H_i)+
\sum_{I\cap U_j\ne \emptyset}(|I\cap U_j|-1).$$
 \end{corollary}
 
In particular, if $\mathcal{H}\cap \Phi =\emptyset$, then $\varrho(G,\mathcal{H}) =0$, and so  we can state the following result, which is a particular case of Corollary \ref{Rocero}.
 
 \begin{remark}
 For any connected graph $G$    of order $n\ge 2$ and any family   $\mathcal{H}=\{H_1,\ldots,H_n\}$   composed by non-empty graphs, 
$$\dim_l(G\circ\mathcal{H})=\displaystyle\sum_{i=1}^{n}\adim_{l}(H_i)+
\sum_{I\cap U_j\ne \emptyset}(|I\cap U_j|-1).$$
 \end{remark}

If $G\cong K_n$, then $\sum_{I\cap U_j\ne \emptyset}(|I\cap U_j|-1)=\max\{0,|I|-1\}$, $|X_E|\in\{0,1\}$, which implies that $\varrho(G,\mathcal{H})=0$, and so Theorem \ref{DimLocalLexicInTermsAdjaLoc} leads to the following.  

\begin{corollary}
For any  integer  $n \ge 2$ and any family   $\mathcal{H}=\{H_1,\ldots,H_n\}$   of graphs, 
$$\dim_l(K_n\circ\mathcal{H})=\displaystyle\sum_{i=1}^{n}\adim_{l}(H_i)+\max\{0,|I|-1\}.$$
Furthermore,   the following assertions hold for a graph $H$.
\begin{itemize}
\item If $H\in \mathcal{G}$, then $\dim_l(K_n\circ H)=n\cdot \adim_{l}(H)+n-1.$

\item If $H\notin \mathcal{G}$, then $\dim_l(K_n\circ H)=n\cdot \adim_{l}(H).$
\end{itemize}
\end{corollary}

Notice that, in the general case,  $\sum_{I\cap U_j\ne \emptyset}(|I\cap U_j|-1)=0$ if and only if each true twin equivalence class of $G$ contains at most one vertex $u_i$ such that $H_i\in \mathcal{G}$.
Thus, we can state the following corollary.

\begin{corollary}
Let $G$ be a connected graph of order $n\ge 2$ and let $\mathcal{H}=\{H_1,\ldots,H_n\}$ be a family of graphs. Then  
$\dim_l(G\circ\mathcal{H})=\displaystyle\sum_{i=1}^{n}\adim_{l}(H_i)$ if and only if for every two adjacent vertices $u_i,u_j\in I$, not belonging to the same true twin equivalence class, there exists $u\in V(G)-(V_E\cup \{u_i,u_j\})$ such that $d_G(u,u_i)\ne d_G(u,u_j)$ and each true twin equivalence class of $G$ contains at most one vertex $u_i$ such that $H_i\in \mathcal{G}$.
\end{corollary}

A particular case of the result above is stated in the next remark. 

\begin{remark}
Let $G$ be a connected bipartite graph of order $n\ge 2$ and let $\mathcal{H}=\{H_1,\ldots,H_n\}$ be a family of graphs. If $\mathcal{H}\not\subseteq \mathcal{G}$, then $$\dim_l(G\circ\mathcal{H})=\displaystyle\sum_{i=1}^{n}\adim_{l}(H_i).$$ 
\end{remark}

\begin{corollary}
Let  $G$ be a connected bipartite graph of order $n$, let $H$ be a non-empty graph, and 
 let  $\mathcal{H}$ be  a family composed by  $n$ graphs. If  $ \mathcal{H}- \Phi=\{H\}$, then
\[\dim_l(G\circ\mathcal{H})=\left\{  \begin{array}{ll}
  \adim_l(H)+1,& \text{if  $H\in \mathcal{G}$};\\
  \\
  \adim_l(H), & \text{otherwise.}
  \end{array}
\right.
  \]
\end{corollary}

\begin{proof}
If $G\cong K_2$, then $\varrho(G,\mathcal{H}) =0$, 
$\displaystyle\sum_{I\cap U_j\ne \emptyset}(|I\cap U_j|-1)=1$ whenever  $H\in \mathcal{G}$, and  $\displaystyle\sum_{I\cap U_j\ne \emptyset}(|I\cap U_j|-1)=0$
 whenever  $H\not \in \mathcal{G}$.
 On the other hand, if $G\not \cong K_2$, then $\displaystyle\sum_{I\cap U_j\ne \emptyset}(|I\cap U_j|-1)=0$,    $\varrho(G,\mathcal{H}) =1$ whenever  $H\in \mathcal{G}$, and  
 $\varrho(G,\mathcal{H}) =0$ whenever  $H\not \in \mathcal{G}$.
Since  in any case $ \displaystyle\sum_{i=1}^{n}\adim_{l}(H_i)=\adim_{l}(H)$, the result follows from Theorem \ref{DimLocalLexicInTermsAdjaLoc}.
\end{proof}

Our next result concerns the case of a family $\mathcal{H}$ composed by empty graphs. %In such a case, $\displaystyle\sum_{i=1}^{n}\adim_{l}(H_i)=0$ and

\begin{remark}\label{DimLocalLex=dimlocalG}
For any connected graph $G$  of order $n\ge 2$ and any family $\mathcal{H}$ composed by $n$  graphs, $$\dim_{l}(G\circ\mathcal{H})\ge\dim_l(G).$$
In particular, if $\mathcal{H}\subset \Phi$, then   $$\dim_{l}(G\circ\mathcal{H})=\dim_l(G).$$
\end{remark}

\begin{proof}
Let $W$ be a local metric basis of $G\circ\mathcal{H}$ and let $W_G=\{u:\;(u,v)\in W\}$ be the projection of $W$ onto $G$. If there exist two adjacent vertices $u_i,u_j\in V(G)-W_G$ not distinguished by any vertex in $W_G$, then no pair of vertices $(u_i,v)\in \{u_i\}\times V(H_i)$, $(u_j,w)\in \{u_j\}\times V(H_j)$ is distinguished by elements of $W$, which is a contradiction. Thus, $W_G$ is a local metric generator for $G$, so $\dim_l(G\circ\mathcal{H})=|W|\ge|W_G|\ge\dim_l(G)$.

Now, we assume that $\mathcal{H}\subset \Phi$ and proceed to show that $\dim_{l}(G\circ\mathcal{H})\le\dim_l(G).$ Let $A$ be a local metric basis of $G$.   For each $H_l\in \mathcal{H}$ we select one vertex $y_l$ and we define the set $A'=\{(u_l,y_l): u_l\in A\}$. Let $(u_i,v)$ and $(u_j,w)$ be two adjacent vertices of  $G\circ\mathcal{H}$. Since $u_i\sim u_j$, there exists $u_l\in A$ such that $d_G(u_i,u_l)\ne d_G(u_j,u_l)$. Now, if $l\ne i,j$,   then we have $d_{G\circ\mathcal{H}}((u_l,y_l),(u_i,v))=d_G(u_i,u_l)  \ne d_G(u_j,u_l)=d_{G\circ\mathcal{H}}((u_l,y_l),(u_j,w))$. If $l=i$, then $d_{G\circ\mathcal{H}}((u_l,y_l),(u_i,v))=2  \ne 1=d_{G\circ\mathcal{H}}((u_l,y_l),(u_j,w))$.
Since the case $l=j$ is analogous to the previous one, we can conclude that $A'$ is  a local metric generator for $G\circ\mathcal{H}$ and, as a consequence, $\dim_{l}(G\circ\mathcal{H})\le \dim_l(G)$. Therefore, the proof is complete.
\end{proof}

In general, the converse of Corollary \ref{DimLocalLex=dimlocalG} does not hold. For instance, we take $G$ as the graph shown in Figure \ref{figExampleNoIfandOnlyif-DimLocalLex=dimlocalG},    $H_1\cong H_5\cong K_2$ and $H_2, H_3, H_4\in \Phi$. In this case, we have that, for instance, $\{u_1,u_5\}$ is a local metric basis of $G$, whereas for any $y\in V(H_1)$ and $y'\in V(H_5)$, the set $\{(u_1,y),(u_5,y')\}$ is a local metric basis of $G\circ {\cal H}$, so  $\dim_{l}(G\circ\mathcal{H})=\dim_l(G)=2$. 

\begin{figure}[h]
\centering
\begin{tikzpicture}[transform shape, inner sep = .7mm]
%%El espaciamiento entre los vértices de arriba
\node [draw=black, shape=circle, fill=black] (v1) at (-2,0) {};
\node [draw=black, shape=circle, fill=white] (v2) at (-1,0) {};
\node [draw=black, shape=circle, fill=white] (v3) at (0,1) {};
\node [draw=black, shape=circle, fill=white] (v4) at (1,0) {};
\node [draw=black, shape=circle, fill=black] (v5) at (2,0) {};
\draw (v1)--(v2)--(v3)--(v4)--(v5);
\draw (v4)--(v2);
\node  at (-2,-0.3) {$u_1$};
\node  at (-1,-0.3) {$u_2$};
\node  at (1,-0.3) {$u_4$};
\node  at (2,-0.3) {$u_5$};
\node  at (0, 1.3) {$u_3$};
\end{tikzpicture}
\caption{The set $\{u_1, u_5\}$ is a local metric basis of this graph.}\label{figExampleNoIfandOnlyif-DimLocalLex=dimlocalG}
\end{figure}
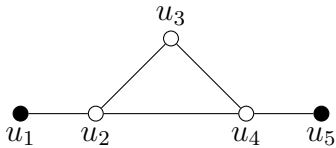

As a direct consequence of Theorems \ref{LocalAdjDimSmall-Large}   and \ref{DimLocalLexicInTermsAdjaLoc}   we deduce the following two results. 

\begin{theorem}
Let $G$ be a connected graph of order $n\ge 2$ and let $\mathcal{H}=\{H_1,\ldots,H_n\}$ be a family composed by non-empty graphs. Then 
$\dim_l(G\circ\mathcal{H})=n$ if and only if each true twin equivalence class of $G$ contains at most one vertex $u_i$ such that $H_i\in \mathcal{G}$  and each  
$H_i\in \mathcal{H}$ is a  bipartite graph having only one non-trivial connected component $H_i^*$ and $r(H_i^*)\le 2$.
\end{theorem}

\begin{theorem}
Let $G$ be a connected true twins free graph of order $n\ge 2$ and let $\mathcal{H}=\{H_1,\ldots,H_n\}$ be a family composed by non-empty graphs of order $n_i$. Then  $\dim_l(G\circ\mathcal{H})=\displaystyle \sum_{i=1}^nn_i-n$ if and only if   $H_i\cong K_{n_i}$, for all $H_i\in \mathcal{H}$.
\end{theorem}

\section{The local adjacency dimension of ${H}$ versus the local metric dimension of $K_1+H$}\label{sectAdjvsK1}

From now on we denote by $\mathcal{G}'$ the set of graphs $H$ satisfying  that there exists a local metric basis of $K_1+H$ which contains the vertex of $K_1$.

\begin{proposition}\label{EquivDimAdjLocaDimLocK_1+H} Let $H$ be a  graph. Then    $H\in \mathcal{G}' $ if and only if  $H\in \mathcal{G} $.
\end{proposition}

\begin{proof}
 Let  $H\in \mathcal{G}' $, and $B$ a  local metric basis  of $\langle u\rangle+H$ such that $u\in B$. Since $u$ does not distinguish any pair of vertices of $H$, $B-\{u\}$ is a local adjacency generator for $H$, and so $\dim_l(\langle u\rangle+H)-1\ge \adim_l(H)$. Now, if there exists a local adjacency basis $A$ of $H$ such that $A\not \subseteq N_H(v)$ for all $v\in V(H)$, then $A$ is a local metric basis of  $\langle u\rangle+H$ and so $\dim_l(\langle u\rangle+H)=\adim_l(H)$, which is a contradiction. Therefore, $H\in \mathcal{G}$.

Now, let $H\in \mathcal{G} $.
Suppose that there exists  a  local metric basis $W$ of $\langle u\rangle+H$  such that $u\not \in W$. In such a case, for every vertex  $x\in V(H)$ there exists $y\in W$ such that $y\not \in N_H(x)$, which implies that $W$ is not a local adjacency basis of $H$, as $H\in \mathcal{G} $.  Thus, since $W$ is a local adjacency generator for $H$, we conclude that  $\dim_l(\langle u\rangle+H)=|W| \ge \adim_l(H)+1$. Therefore, for any local adjacency basis $A$ of $H$, $A\cup \{u\}$ is a  local adjacency basis of $\langle u\rangle+H$.
\end{proof}

\begin{theorem}{\rm \cite{RV-F-2013}} \label{RelacAdimLocal-LocalDimK1+H}
Let $H$ be a non-empty graph. The following assertions hold.
\begin{enumerate}[{\rm (i)}]
\item  If $H\not \in \mathcal{G}'$, then 
$\adim_l( H)=  \dim_l(K_1+H).$

\item   If  $H  \in \mathcal{G}'$, then 
$\adim_l( H)= \dim_l(K_1+H)-1.$

\item If $H$ has radius $r(H)\ge 4$, then 
 $\adim_l(  H)=   \dim_l(K_1+H).$
\end{enumerate}
\end{theorem}

As the following result shows, we can express all our previous results in terms of the local adjacency dimension of the graphs $K_1+H_i$, where    $H_i\in \mathcal{H}$, i.e., Theorem \ref{DimLocalLexicInTermsK_1+H} is analogous  to Theorem \ref{DimLocalLexicInTermsAdjaLoc}.

\begin{theorem}\label{DimLocalLexicInTermsK_1+H}
Let  $G$ be a connected graph     of order $n\ge 2$,  and $\mathcal{H}=\{H_1,\ldots,H_n\}$  a family   of  graphs.
 Then
$$\dim_l(G\circ\mathcal{H})=\displaystyle\sum_{i=1}^{n}\dim_{l}(K_1+H_i)-
\tau+\varrho(G,\mathcal{H}),$$
where $\tau$ is the number of   non-singleton true twin equivalence classes of $G$ having at least one vertex $u_i$ such that 
 $H_i\in \mathcal{G}' $ .
\end{theorem}

\begin{proof}
 Notice that, by Proposition  \ref{EquivDimAdjLocaDimLocK_1+H}, the parameter $\varrho(G,\mathcal{H})$  can be redefined in terms of  $\mathcal{G}'$. 
The result immediately follows from 
 Proposition \ref{EquivDimAdjLocaDimLocK_1+H}  and Theorems \ref{DimLocalLexicInTermsAdjaLoc}  and  \ref{RelacAdimLocal-LocalDimK1+H}.
\end{proof}

\begin{lemma}\label{lemmaGirth}
Let $H$ be a connected  graph different from a tree. If $H\in \mathcal{G}$, then $\mathtt{g}(H)\le 6$. 
\end{lemma}

\begin{proof}
Let $A$ be local adjacency basis of $H$. 
 Since $H\in \mathcal{G}$, we consider $v$ as the vertex of $H$ such that $A\subseteq N_H(v)$. Let $N_i(v)=\{u\in V(H):\; d_H(v,u)=i\}$. 
Since $A\subseteq N_1(v)$, we have that $N_3(v)$ is an independent set and $N_i(v)=\emptyset$, for all $i\ge 4$. Therefore, 
 $\mathtt{g}(H)\le 6$.
\end{proof}
By Proposition \ref{EquivDimAdjLocaDimLocK_1+H}, Theorem \ref{DimLocalLexicInTermsK_1+H}   and Lemma \ref{lemmaGirth} we can derive the following consequence of Theorem \ref{DimLocalLexicInTermsK_1+H} (or equivalently, Theorem \ref{DimLocalLexicInTermsAdjaLoc}).

\begin{corollary}\label{CoroFirstInequalityTight}
Let  $G$ be a connected graph   of order $n\ge 2$, and $\mathcal{H}=\{H_1,\ldots,H_n\}$ a family composed by connected graphs. If   each $H_i\in \mathcal{H}$  has radius $r(H_i)\ge 4$, or  $H_i$ is not  a tree and it has girth $\mathtt{g}(H_i)\ge 7$, then 
$$\dim_l(G\circ\mathcal{H})=\displaystyle\sum_{i=1}^{n}\dim_{l}(K_1+H_i)=\sum_{i=1}^{n}\adim_{l}(H_i).$$
\end{corollary}

%%%%%%%%%%
\begin{proposition} {\rm \cite{RV-F-2013}}   \label{DimAdLocCycle}
For any integer $n\ge 4$,  $\adim_l(  C_n)= \left\lceil\frac{n}{4}\right\rceil$.
\end{proposition}
%%%%%%%%%

From Corollary \ref{CoroFirstInequalityTight} and Proposition \ref{DimAdLocCycle} we deduce the following result. 

\begin{proposition}
Let  $G$ be a connected graph   of order $t\ge 2$, and $\mathcal{H}=\{C_{n_1},\ldots,C_{n_t}\}$ a family composed by cycles of order at least $7$. Then
$$\dim_l(G\circ\mathcal{H})=\displaystyle\sum_{i=1}^{t} \left\lceil\frac{n_i}{4}\right\rceil.$$
\end{proposition}

\section{On the local adjacency dimension of 
lexicogra\-phic product graphs}\label{sectAdjProduct}

%%%%%%%%%%

By a simple transformation of Theorem \ref{DimLocalLexicInTermsAdjaLoc} we obtain an analogous result on the local adjacency dimension of 
lexicographic product graphs, which we will state without proof. To this end, we consider again some of our previous notation. As above, let $\{U_1,U_2,\dots , U_k\}$ be the set of non-singleton true twin equivalence classes of a connected graph $G$ of order $n\ge 2$, and let $\mathcal{H}=\{H_1,\ldots,H_n\}$ be a family of graphs. Recall that $V_E=\left\{u_i\in V(G)-T(G) :\, H_i \in \Phi \right\}$, 
$I=\{u_i\in V(G):\, H_i\in \mathcal{G}\}$ and, for any $I_j=I\cap U_j\ne \emptyset $, we can choose some  $u\in I_j$ and set  $I'_j=I_j-\{u\}$. Moreover, recall that $X_E=I-\bigcup_{I_j'\ne \emptyset}I_j'$. Now, we say that two vertices $u_i,u_j \in X_E$ satisfy the relation $\mathcal{R}'$ if and only if $u_i \sim u_j$ and $d_{G,2}(u,u_i)=d_{G,2}(u,u_j)$ for all $u\in V(G)-(V_E \cup \{u_i,u_j\})$. We define $\mathcal{A}'$ as the family of sets $A\subseteq X_E$  such that for every pair of vertices $u_i,u_j\in X_E$ satisfying $\mathcal{R}'$ there exists a vertex in $A$ that distinguishes them. Finally, we define $\varrho'(G,\mathcal{H})=\displaystyle\min_{A\in \mathcal{A}'}\left\{|A|\right\}.$

\begin{theorem}\label{AdjDimLocalLexicInTermsAdjaLoc}
Let  $G$ be a connected graph     of order $n\ge 2$, let $\{U_1,U_2,\dots , U_k\}$ be the set of non-singleton true twin equivalence classes of $G$ and let $\mathcal{H}=\{H_1,\ldots,H_n\}$ be a family     of graphs. 
Then
$$\adim_l(G\circ\mathcal{H})=\displaystyle\sum_{i=1}^{n}\adim_{l}(H_i)+
\sum_{I\cap U_j\ne \emptyset}(|I\cap U_j|-1)+\varrho'(G,\mathcal{H}).$$
\end{theorem}

Let $G\cong P_4$ where $V(P_4)=\{u_1,u_2,u_3,u_4\}$ and $u_i\sim u_{i+1}$, for $i\in \{1,2,3\}$. If $H_1\cong H_2\cong H_4 \cong P_3$ and  $H_3\cong N_3$, then  $\dim_l(G\circ\mathcal{H})=3 <4=\adim_{l}(G\circ\mathcal{H})$. Notice that $\varrho(G,\mathcal{H})=0$ and $\varrho'(G,\mathcal{H})=1$. However, 
if $ H_2\cong H_3 \cong P_3$ and  $H_1\cong H_4\cong N_3$, then $\varrho(G,\mathcal{H})=\varrho'(G,\mathcal{H})=1$ and $\dim_l(G\circ\mathcal{H})=3 =\adim_{l}(G\circ\mathcal{H})$.
%despite
%the premises of Theorem \ref{lexiSame_local-Metric_Adj_dim} (ii) fail, as $u_2$ and $u_3$ are adjacent but not true twins, $H_2\in \mathcal{G}$ and $H_3\in \mathcal{G}$, and $H_1$ and is $H_4$ empty, being $\{u_1,u_4\}=N_G(u_1)\triangledown N_G(u_2)$. In this case, $\dim_l(G\circ\mathcal{H})=\adim_{l}(G\circ\mathcal{H})=3\ne 2=\displaystyle\sum_{i=1}^{n}\adim_{l}(H_i)+
%\sum_{I\cap U_j\ne \emptyset}(|I\cap U_j|-1)$, and $\varrho(G,\mathcal{H})=\varrho'(G,\mathcal{H})=1$.

We already know that for any graph $G$ of diameter less than or equal to two, $\dim_l(G)=\adim_l(G)$. However, the previous example shows that 
 the above mentioned equality is not restrictive to graphs of diameter at most two, as $D(G\circ\mathcal{H})=D(P_4)=3$.

Notice that $\varrho'(G,\mathcal{H})\ge\varrho(G,\mathcal{H})$, which is a direct consequence of Theorems~\ref{DimLocalLexicInTermsAdjaLoc} and~\ref{AdjDimLocalLexicInTermsAdjaLoc}, as well as the fact that $\adim_l(G)\ge\dim_l(G)$ for any graph $G$. The next result corresponds to the case $\varrho(G,\mathcal{H})=\varrho'(G,\mathcal{H})$.

\begin{theorem}\label{FangoIguales}
Let  $G$ be a connected graph     of order $n\ge 2$,  and $\mathcal{H}=\{H_1,\ldots,H_n\}$ a family     of  graphs.  Then
$\dim_l(G\circ\mathcal{H})=\adim_l(G\circ\mathcal{H})$ if and only if $\varrho(G,\mathcal{H})=\varrho'(G,\mathcal{H})$.
\end{theorem}

We now characterize the case $\varrho(G,\mathcal{H})=\varrho'(G,\mathcal{H})=0$.  The symmetric difference of two sets $U$ and $W$ will be denoted by $U \triangledown W$. 

\begin{theorem}\label{lexiSame_local-Metric_Adj_dim}
Let $G$ be a connected graph of order $n\ge 2$ and let $\mathcal{H}=\{H_1,\ldots,H_n\}$ be a family of graphs. Then the following assertions are equivalent.
\begin{enumerate}[{\rm (i)}]
\item $\dim_l(G\circ\mathcal{H})=\adim_{l}(G\circ\mathcal{H})=\displaystyle\sum_{i=1}^{n}\adim_{l}(H_i)+
\sum_{I\cap U_j\ne \emptyset}(|I\cap U_j|-1).$
\item For any pair of adjacent vertices $u_i,u_j\in V(G)$, not belonging to the same true twin equivalence class,  $H_i\notin \mathcal{G}$ or $H_j\notin \mathcal{G}$, or there exists $u_l\in N_G(u_i)\triangledown N_G(u_j)$ where $H_l$ is not empty.
\end{enumerate}
\end{theorem}

\begin{proof}
By Theorems \ref{DimLocalLexicInTermsAdjaLoc}, \ref{AdjDimLocalLexicInTermsAdjaLoc} and \ref{FangoIguales}, we only need to show that $\varrho'(G,\mathcal{H})=0$ if and only if  (ii) holds.

\noindent 
$((i)\Rightarrow (ii))$ If  $\varrho'(G,\mathcal{H})=0$, then for every two adjacent vertices $u_i,u_j\in I$, not belonging to the same true twin equivalence class, there exists $u_l\in V(G)-(V_E\cup \{u_i,u_j\})$ such that $d_{G,2}(u_l,u_i)\ne d_{G,2}(u_l,u_j)$, which implies that   $u_l\in N_G(u_i)\triangledown N_G(u_j)$ and  $H_l$ is not empty. Now, if $u_i,u_j\not \in I$,  then   $H_i\notin \mathcal{G}$ or $H_j\notin \mathcal{G}$.\\

\noindent
$((ii)\Rightarrow (i))$ If for any pair of adjacent vertices $u_i,u_j\in V(G)$, not belonging to the same true twin equivalence class,  $H_i\notin \mathcal{G}$ or $H_j\notin \mathcal{G}$, or there exists $u_l\in N_G(u_i)\triangledown N_G(u_j)$ where $H_l$ is not empty, then no pair of adjacent vertices satisfy $\mathcal{R}'$ and $V(G)-X_E$ is a local adjacency generator for $G$, which implies that  $\varrho'(G,\mathcal{H})=0$.
\end{proof}

\end{document}